\documentclass[11pt]{article}

\usepackage[english]{babel}
\usepackage[utf8]{inputenc}

\usepackage{amsmath,amssymb,amsthm}
\usepackage[linktocpage,unicode]{hyperref}

\usepackage{tikz,rotating}

\usepackage{a4wide}
\usepackage{caption}
\usepackage{mathabx}

\hypersetup{
    pdftoolbar=true,        % show Acrobat toolbar?
    pdfmenubar=true,        % show Acrobat menu?
    pdffitwindow=false,     % window fit to page when opened
    pdfstartview={FitH},    % fits the width of the page to the window
    pdftitle={Δ-cumulants in terms of moments},
    pdfauthor={Matthieu Josuat-Vergès},
    pdfsubject={},
    pdfkeywords={},
    pdfnewwindow=true,      % links in new window
    colorlinks=true,       % false: boxed links; true: colored links
    linkcolor=red,          % color of internal links
    citecolor=blue,        % color of links to bibliography
    filecolor=magenta,      % color of file links
    urlcolor=cyan,           % color of external links
    unicode=true          % non-Latin characters in Acrobat's bookmarks
}

\newtheorem{theo}{Theorem}[section]

\newtheorem{lemm}[theo]{Lemma}
\newtheorem{prop}[theo]{Proposition}

\theoremstyle{definition}
\newtheorem{defi}[theo]{Definition}
\newtheorem{rema}[theo]{Remark}

\newcommand{\inv}{\langle -1 \rangle}

\DeclareMathOperator{\NC}{NC}
\DeclareMathOperator{\IN}{IN}
\newcommand{\ST}{\mathcal{S}}
\newcommand{\PST}{\mathcal{S}'}
\DeclareMathOperator{\wt}{wt}
\DeclareMathOperator{\std}{std}
\DeclareMathOperator{\arc}{arc}
\DeclareMathOperator{\cov}{cov}
\DeclareMathOperator{\intern}{int}
\DeclareMathOperator{\Res}{Res}

\newcommand{\dconv}{\mathrel{\boxtriangleup}}
\newcommand{\mconv}{\mathrel{\boxvoid}}

\title{$\Delta$-cumulants in terms of moments}

\author{Matthieu Josuat-Vergès  \thanks{Supported by ANR CARMA (ANR-12-BS01-0017).}  \\
 CNRS and Institut Gaspard Monge \\
 Université Paris-Est Marne-la-Vallée \\
 France \\
 \small \tt matthieu.josuat-verges@u-pem.fr
}

\date{}

\usetikzlibrary{arrows}

\begin{document}
      
\maketitle

\begin{abstract}
The $\Delta$-convolution of real probability measures, introduced by Bożejko,
generalizes both free and boolean convolutions. It is linearized by the
$\Delta$-cumulants, and Yoshida gave a combinatorial formula for moments
in terms of $\Delta$-cumulants, that implicitly defines the latter.
It relies on the definition of an appropriate weight on noncrossing partitions.
We give here two different expressions for the $\Delta$-cumulants: the first one
is a simple variant of Lagrange inversion formula, and the second one is
a combinatorial inversion of Yoshida's formula involving Schröder trees.
\end{abstract}

\tikzset{every picture/.style={scale=0.6}  }

\section{Introduction}

The additive convolution $\mu\ast\nu$ of two real probability measures $\mu$ and $\nu$ is usually defined as the law of 
$X+Y$ where $X$ and $Y$ are  two independent random variables of law $\mu$ and $\nu$, respectively. 
Other operations, that we can see as deformed convolutions, are obtained by replacing the classical notion of independence with 
other ones coming from noncommutative probability theories. Two important examples are the free convolution $\mu\boxplus\nu$
(see Voiculescu \cite{voiculescu}) and the boolean convolution $\mu\biguplus\nu$ (see Speicher and Woroudi \cite{speicherworoudi}).

The $\Delta$-convolution was introduced by Bożejko \cite{bozejko} as a special case of the conditionally free convolution 
from \cite{bozejkoleinertspeicher}, and further studied in \cite{bozejkokrystekwojkowski,krystekyoshida,yoshida}.
See \cite{lehner} for the general context.
This operation, denoted $\boxtriangleup\,$, depends on another measure $\omega$ and specializes at 
$\boxplus$ (respectively, $\biguplus$) when $\omega$ is the Dirac distribution at $1$ (respectively, $0$).
It can be defined analytically as follows. First, $\mu$ is characterized its Cauchy transform:
\[
   G_\mu(z) = \int\frac{1}{z-x} \mu( {\rm d}x ).
\]
Then, the function $R^{\Delta}_{\mu}(z)$ is implicitly defined by:
\begin{equation} \label{rdtransform}
   G_\mu(z) = \frac{1}{z - R^{\Delta}_{\mu}( G_{\mu \,\mconv\, \omega} (z) ) } 
\end{equation}
where $\mu\mconv\omega$ is the multiplicative convolution of $\mu$ and $\omega$
(defined like the additive convolution above but with $XY$ instead of $X+Y$).
This function $R^{\Delta}_{\mu}(z)$ also characterizes the measure $\mu$, and it is called its $R^\Delta$-transform.
This is a deformation of Voiculescu's $R$-transform (itself being the analog of the logarithm of the Fourier transform in 
classical probability), see \cite{voiculescu}. Then $\dconv$ is characterized by the fact that it is linearized by 
the $R^\Delta$-transform:
\[ 
  R^{\Delta}_{ \mu \,\dconv\, \nu }(z) = R^{\Delta}_{ \mu}(z) + R^{\Delta}_{ \nu } (z).  
\]
See \cite{yoshida} for details.

Rather than the functional equation in \eqref{rdtransform}, one can consider the relations between the 
moments $M_n(\mu)$ and $\Delta$-cumulants $C_n(\mu)$, which are the coefficients in the following expansions,
when they exist:
\[
   G_\mu(z) = \sum_{n=0}^{\infty}  \frac{M_n(\mu)}{z^{n+1}}, \qquad
  R^{\Delta}_{\mu}(z) = \sum_{ n = 1 }^{ \infty }  C_n(\mu) z^{n-1}
\]
near $z=\infty$ and $z=0$, respectively. Alternatively, we have $M_n(\mu) = \int x^n \mu({\rm d}x)$. These relations 
depend on the moments of $\omega$, denoted $\delta_n = M_n(\omega)$, that we also assume to exist.
Yoshida \cite{yoshida} proved that for an appropriate weight function $\wt$ on the set $\NC_n$ of noncrossing partitions
of $\{1,\dots,n\}$ (defined in the next section), we have:
\begin{align} \label{relMC}
  M_n(\mu) = \sum_{\pi \in \NC_n} C_{\pi}(\mu) \wt(\pi), \quad \text{ where } C_{\pi}(\mu) = \prod_{ B \in \pi } C_{\# B}(\mu).
\end{align}
This generalizes the free and boolean cases, where we have an unweighted sum over noncrossing partitions, and interval partitions,
respectively. But in the weighted case of \eqref{relMC}, inverting the relation cannot be done via a Möbius inversion of a poset, 
since the weight $\wt(\pi)$ depends on the $\delta_i$'s.

In this work, we provide two different expressions for the $\Delta$-cumulants. The first one (Theorem~\ref{theolagrange}) is 
based on the functional equation in \eqref{rdtransform}, and is a variant of Lagrange inversion formula (see \cite{comtet})
where a Hadamard product is involved. The second one (Theorem~\ref{theotrees}) is a combinatorial formula that is the inverse 
of \eqref{relMC}, proved by inverting a matrix which is the multiplicative extension of Yoshida's weight.
%(unlike in the free and boolean cases, this matrix inversion cannot be seen as Möbius inversion). DEJA DIT
The solution is in terms of Schröder trees, and relies on related notions taken from \cite{josuatvergesmenousnovellithibon}. 
In that work, Schröder trees appeared naturally because the relations between moments and free cumulants are interpreted in the group of 
an operad of trees, or also in terms of characters of Hopf algebras of trees (building on \cite{ebrahimifardpatras1,ebrahimifardpatras2}). 
However, we don't have such an algebraic interpretation for the case of $\Delta$-cumulants.

\section{Definitions}

When $\pi$ is a set partition of a set $X$, we denote $\overset{\pi}{\sim}$ the equivalence relation defined by
$ i \overset{\pi}{\sim} j $ iff $i,j\in B$ for some block $B\in\pi$.

% We denote $i \overset{\pi}{\sim} j$ to say that $i$ and $j$ are two elements of a block of $\pi$. The relation $\overset{\pi}{\sim}$
% is an equivalence relation.
% 

Let $\NC_n$ denote the set of {\it noncrossing partitions} of $\{1,\dots,n\}$, i.e.~set partitions of $\{1,\dots,n\}$ where there 
exist no $i,j,k,l$ such that $i<j<k<l$, $i \overset{\pi}{\sim} k$, $j \overset{\pi}{\sim} l$ and $j \overset{\pi}{\nsim} k$.
For example, $\{\{1,4,6\},\{2,3\},\{5\}\} \in\NC_6$. To lighten the notation, the same is written $146|23|5$, and it 
is represented as:
\begin{equation} \label{ncexample}
  \begin{tikzpicture}
   \tikzstyle{ver} = [circle, draw, fill, inner sep=0.5mm]
   \tikzstyle{edg} = [line width=0.6mm]
   \node      at (1,-0.5) {\small 1};
   \node      at (2,-0.5) {\small 2};
   \node      at (3,-0.5) {\small 3};
   \node      at (4,-0.5) {\small 4};
   \node      at (5,-0.5) {\small 5};
   \node      at (6,-0.5) {\small 6};
   \pgfresetboundingbox
   \node[ver] at (1,0) {};
   \node[ver] at (2,0) {};
   \node[ver] at (3,0) {};
   \node[ver] at (4,0) {};
   \node[ver] at (5,0) {};
   \node[ver] at (6,0) {};
   \draw[edg] (1,0) .. controls (2.5,1.7) .. (4,0);
   \draw[edg] (4,0) .. controls (5,1) .. (6,0);
   \draw[edg] (2,0) .. controls (2.5,0.5) .. (3,0);
  \end{tikzpicture}\; .
\end{equation}
\vspace{-2mm}

Endowed with the reverse refinement order, $\NC_n$ is a lattice, first defined by Kreweras~\cite{kreweras}.
Its minimal element is $0_n=1|2|\dots|n$,  and its maximal element is $1_n=12 \dots n$.

Let $\arc(\pi)$ denote the set of {\it arcs} of $\pi\in\NC_n$, i.e.~pairs $(i,j)$ such that $i<j$, $i \overset{\pi}{\sim}j$,
and there is no $k$ such that $i<k<j$ and $i \overset{\pi}{\sim}j \overset{\pi}{\sim}k$.
They are indeed arcs in the graphical representations as in \eqref{ncexample}, for example those of $146|23|5$ are
$(1,4)$, $(4,6)$, $(2,3)$.
Note that there are $\#B-1$ arcs inside a block $B$ of $\pi$, and it follows that $\#\arc(\pi) + \# \pi = n$.

\begin{defi}
Yoshida's weight $\wt(\pi)$ of $\pi\in\NC_n$ is:
\begin{align}  \label{defwt}
   \wt(\pi) = \prod_{ (i,j) \,\in\, \arc(\pi) } \delta_{j-i-1}.
\end{align}
\end{defi}

\begin{rema}
If we allow $\omega$ to be a positive measure (i.e.~$\delta_0=M_0(\omega)$ is any positive number instead of $\delta_0=1$
for a probability measure), we see in Equation \eqref{relMC} that $M_n(\mu)$ is homogeneous of degree $n$ in $C_1(\mu),C_2(\mu),\dots$ 
and $\delta_0,\delta_1,\dots$ (by the relation $\#\pi + \# \arc(\pi)=n$).
So there is no loss of generality when we assume $\delta_0 = 1$.
\end{rema}

Let $\IN_n\subset \NC_n$ denote the set of {\it interval partitions} of $\{1,\dots,n\}$, i.e.~set partitions where each block is an 
interval of consecutive integers. Equivalently, $\pi\in\NC_n$ is in $\IN_n$ iff it has no arc $(i,j)$ with $j-i\geq 2$.

In the free case ($\delta_i=1$ for all $i$), we have $\wt(\pi)=1$ for all $\pi\in\NC_n$. So Equation~\eqref{relMC}
is the known relation for free cumulants \cite{speicher}. In the boolean case ($\delta_1=1$ and $\delta_i=0$ for $i\geq2$),
we have $\wt(\pi)=1$ if $\pi\in\IN_n$ and $0$ otherwise. So Equation~\eqref{relMC} is the known relation for boolean
cumulants \cite{speicherworoudi}.

%\section{\texorpdfstring{Generating functions of moments and $\Delta$-cumulants}{Generating functions of moments and Δ-cumulants}}

The {\it Hadamard product} $\odot$ of two series is defined by:
\begin{align*}
  \Big( \sum a_n z^n \Big) \odot \Big( \sum b_n z^n \Big) = \Big( \sum a_n b_n  z^n \Big).
\end{align*}
This operation makes sense either for formal power series, or functions that are analytic at a specified point.
If two measures $\mu$ and $\nu$ have all their moments, their Cauchy transforms are analytic near $z=\infty$,
and we have:
\begin{equation} \label{hadamardcauchy}
   G_{\mu \,\mconv\,  \nu }(z) = G_{\mu}(z) \odot G_{\nu}(z).
\end{equation}
Indeed, let $X$ and $Y$ be two independent random variables of law $\mu$ and $\nu$, respectively. Then we have
$ \mathbb{E}[(XY)^n] = \mathbb{E}[ X^n Y^n ] = \mathbb{E}[X^n] \mathbb{E}[Y^n]$, 
so $M_n(\mu \mconv \nu) = M_n(\mu) M_n(\nu)$.

From now on, we write $M_n$ for the moments and $C_n$ for the $\Delta$-cumulants, dropping the dependence in $\mu$,
and consider their generating functions:
\begin{align*}
  M(z) = \sum_{n\geq 0 } M_n z^{n+1}, \qquad C(z) = \sum_{n\geq 1} C_n z^{n-1}.
\end{align*}
And to avoid confusion, we take specific notations for the two specializations of $C_n$:
$F_n$ and $B_n$ are respectively the free cumulants and boolean cumulants associated with $M_n$.
Their generating functions are:
\[
   F(z) = \sum_{n\geq 1} F_n z^{n-1}, \qquad  B(z) = \sum_{n\geq 1} B_n z^{n-1}.
\]
Moreover, the generating function of the moments of $\omega$ is similar to $M(z)$:
\[
  \Delta(z) = \sum_{n\geq 0 } \delta_n z^{n+1}.
\]
In fact, since our results are essentially of algebraic or combinatorial nature, we don't need to assume that 
$(C_n)$ or $(\delta_n)$ is the moment sequence of some measure, we can treat them as formal variables and their
generating functions as formal power series. In this setting, $(M_n)$ and $(C_n)$ are related by \eqref{relMC}
if and only if their generating functions are related by:
\begin{align} \label{relMCz}
  M(z) = \frac{ z }{ 1 - z C( M(z)\odot \Delta(z) ) }.
\end{align} 
This is a rewriting of \eqref{rdtransform}, using \eqref{hadamardcauchy} and changing $z$ to $z^{-1}$.

In particular, in the free case we have $M(z)\odot \Delta(z) = M(z)$, and the relation between $M(z)$ and $F(z)$
is the definition of Voiculescu's $R$-transform \cite{voiculescu}:
\begin{equation}  \label{relMF}
  M(z) = \frac{ z }{ 1 - zF(M(z)) }.
\end{equation}
And in the boolean case, $M(z)\odot \Delta(z) = z$, and we recover the analytic definition of boolean cumulants \cite{speicherworoudi}:
\begin{equation} \label{relMB}
  B(z) = \frac{1}{z} - \frac{1}{M(z)}.
\end{equation}

\section{Lagrange inversion for cumulants}

In this section, we denote by $[z^k] g(z)$ the coefficient of $z^k$ in a formal Laurent series $g(z)$.
A formal power series $f(z) = \sum_{n\geq 1}  a_n z^n $ with $a_1 \neq 0$ has a unique compositional inverse
$f^{\inv}(z)$, such that $f( f^{\inv} (z) ) = f^{\inv}( f(z) ) = z $.
Lagrange inversion formula is the identity:
\begin{equation} \label{lagrange}
   [z^n] f^{\inv}(z) = \frac 1n [z^{n-1}] \Big(\frac{ z  }{ f(z) }\Big)^n.
\end{equation}
It comes in a wide range of different forms and has a lot of variants and generalizations,
see Comtet's book \cite[Chapter~III]{comtet}. 
Let us review how to use it in the case of free cumulants, following Speicher~\cite{speicher}.
From \eqref{relMF}, we get:
\[
  M(z) - zM(z)F(M(z)) = z,
\]
and then:
\[
  \frac{ M(z) }{ 1 + M(z)F(M(z))  } = z, 
\]
i.e.~$M^{\inv}(z) = \frac{z}{1+zF(z)} $. Applying \eqref{lagrange} gives
\[
  M_n = \frac{1}{n+1} [z^n] \big( 1 + zF(z) \big)^{n+1}.
\]

Another identity is Hermite's formula \cite[p. 150, Theorem D)]{comtet}:
\[
  [z^n] \frac{z}{ f^{\inv}(z) } = [z^n] f'(z) \Big(\frac{ z }{ f(z) }\Big)^n,
\]
from which we get:
\[
  F_n = [z^n] M'(z) \Big(\frac{z}{M(z)}\Big)^n. %= [z^0]  \frac{M'(z)}{M(z)^n}.
\]
One might prefer a formula only involving $M(z)$ and not its derivative, and this is possible at the condition 
of working with Laurent series. Indeed, the previous formula also gives:
\[
  F_n = [z^0]  \frac{M'(z)}{M(z)^n},
\]
and since $\frac{M'(z)}{M(z)^n} = - \frac{1}{n-1}  (\frac{1}{M(z)^{n-1}})'$ for $n\geq2$, we have:
\begin{equation} \label{lagrangeFn}
  F_n = - \frac{1}{n-1}  [ z ] \frac{1}{M(z)^{n-1}}.
\end{equation}

In the case of our series related by Equation~\eqref{relMCz}, we can adapt a classical proof of Lagrange inversion formula to get the following:

\begin{theo} \label{theolagrange}
For $n\geq2$, the $n$th $\Delta$-cumulant is given by:
\begin{equation} \label{lagrangeCn}
 C_n = \frac{1}{n-1} [z^{-1}] \frac{  \dfrac{M'(z)}{M(z)^2}  -  \dfrac{1}{z^2}  }{  \big( M(z)\odot \Delta(z) \big)^{n-1}   }.
\end{equation}
\end{theo}

\begin{proof}
From \eqref{relMCz}, we have:
\[
  \frac{ M(z)-z }{ zM(z) } = \sum_{n\geq1} C_n \big( M(z)\odot \Delta(z) \big)^{n-1}.
\]
Taking the derivative, we have:
\[
  -\frac{1}{z^2} + \frac{M'(z)}{M(z)^2} =  \sum_{n\geq2} (n-1) C_n \big( M(z)\odot \Delta(z) \big)' \big( M(z)\odot \Delta(z) \big)^{n-2}. 
\]
Divide on both sides by $ \big( M(z)\odot \Delta(z) \big)^{k} $ to get:
\[
  \frac{  \dfrac{M'(z)}{M(z)^2}  -  \dfrac{1}{z^2}  }{  \big( M(z)\odot \Delta(z) \big)^{k}   }
  =  
  \sum_{n\geq2} (n-1) C_n \big( M(z)\odot \Delta(z) \big)' \big( M(z)\odot \Delta(z) \big)^{n-k-2}. 
\]
Then, take the coefficient of $z^{-1}$. To deal with the right hand side, note that 
if $n \neq k+1$, we have:
\[
  \big( M(z)\odot \Delta(z) \big)' \big( M(z)\odot \Delta(z) \big)^{n-k-2} = \frac{1}{n-k-1} \big( (M(z)\odot \Delta(z)) ^{n-k-1} \big)'.
\]
Since $[z^{-1}] f'(z) = 0 $ for any Laurent series $f(z)$, it remains:
\[
  [z^{-1}] \frac{  \dfrac{M'(z)}{M(z)^2}  -  \dfrac{1}{z^2}  }{  \big( M(z)\odot \Delta(z) \big)^{k}   }
  =  
  k C_{k+1} [z^{-1}] \big( M(z)\odot \Delta(z) \big)' \big( M(z)\odot \Delta(z) \big)^{-1}.
\]
From $ M(z)\odot \Delta(z) = z + O(z^2)$, we easily obtain  
$ [z^{-1}] \big( M(z)\odot \Delta(z) \big)' \big( M(z)\odot \Delta(z) \big)^{-1}=1$.
We thus obtain a formula for $k C_{k+1}$ and Equation~\eqref{lagrangeCn} follows.
\end{proof}

We end this section by a few remarks about the previous theorem.
In the boolean case, we have $\Delta(z)=z$ and $ M(z)\odot\Delta(z) = z$, so it says:
\[
  (n-1) B_n = [ z^{-1} ] \Big( \frac{M'(z)}{ z^{n-1} M(z)^2} - \frac{1}{z^{n+1}} \big) = [z^{n-2}] \frac{M'(z)}{M(z)^2}.
\]
After multiplying by $z^{n-2}$ and summing for $n\geq2$, we get:
\[
  B'(z) = \frac{M'(z)}{M(z)^2} - \frac{1}{z^2}
\]
where the term $-\frac{1}{z^2}$ is needed to remove negative powers of $z$ from $\frac{M'(z)}{M(z)^2}$.
This agrees with the analytic definition of boolean cumulants in \eqref{relMB}.

In the free case, $ M(z)\odot\Delta(z) = M(z)$, so we get:
\[
  F_n = \frac{1}{n-1} [z^{-1}] \Big( \frac{ M'(z)  }{ M(z)^{n} } - \frac{1}{z^2M(z)^{n-1}}  \Big).
\]
Since $\frac{ M'(z)  }{ M(z)^{n} } = ( - \frac{1}{(n-1)M(z)^{n-1}} )'$, we have $ [z^{-1}] \frac{ M'(z)  }{ M(z)^{n} } = 0 $. 
So the formula gives
\[
  F_n = - \frac{1}{n-1} [z^{-1}]  \frac{1}{z^2M(z)^{n-1}}  
\]
and we recover \eqref{lagrangeFn}.

It is worth writing the previous theorem in a more analytic way, using Cauchy transforms. We have:
\[
  (n-1) C_n = [z^{-1}] \frac{  \dfrac{M'(z)}{M(z)^2}  -  \dfrac{1}{z^2}  }{  \big( M(z)\odot \Delta(z) \big)^{n-1}   }
            = [ z ] \frac{  \dfrac{M'(\frac 1z)}{M(\frac 1z)^2}  - z^2 }{  \big( M(\frac 1z)\odot \Delta(\frac 1z) \big)^{n-1}   }
            = [z^{-1}] \frac{  \dfrac{M'(\frac 1z)}{z^2M(\frac 1z)^2}  - 1 }{  \big( M(\frac 1z)\odot \Delta(\frac 1z) \big)^{n-1}   }.
\]
Since $M(\frac 1z) = G_{\mu}(z) $, and $\Delta(\frac 1z)=G_{\omega}(z)$, this gives:
\[
  (n-1) C_n = [z^{-1}] \frac{ - \dfrac{G_{\mu}'(z)}{ G_{\mu}(z)^2 } - 1 }{ (G_{\mu}(z) \odot G_{\omega}(z))^{n-1} }.
\]
For a function which is analytic near $z=\infty$, its residue at $z=\infty$ is given by $\Res_{\infty} f(z) = -[z^{-1}] f(z)$
and can be calculated by a contour integral. So the analytic formulation of the previous theorem is:
\[
  C_n = \frac{1}{n-1} \Res_{\infty} \frac{  \dfrac{G_\mu'(z)}{G_\mu(z)^2}  +  1  }{  G_{\mu \,\mconv\, \omega }(z)^{n-1}   }.
\]

We do not know if there exists another variant of Lagrange inversion that would give the moments $M_n$ in terms
of $C(z)$ and $\Delta(z)$. 

\section{Inverting the relation}

We now present how to inverse the relation in Equation~\eqref{relMC} to get a formula for $C_n$ in terms of $M_1,\dots,M_n$.
For small values of $n$, \eqref{relMC} gives:
\begin{align*}
   M_1 &= C_1, \\
   M_2 &= C_2 + C_1^1,  \\
   M_3 &= C_3 + (2+\delta_1) C_2 C_1 + C_1^3,  \\
   M_4 &= C_4 + (2+2\delta_1) C_3 C_1 + (1+\delta_2) C_2^2 + (3+2\delta_1+\delta_2)C_2C_1^2 + C_1^4.
\end{align*}
From that, we successively get the values:
\begin{align*}
   C_1 &= M_1, \\
   C_2 &= M_2 - M_1^1,  \\
   C_3 &= M_3 - (2+\delta_1)M_2M_1 + (1+\delta_1)M_1^3, \\
   C_4 &= \scriptstyle{M_4 - (2+2\delta_1)M_3M_1 - (1+\delta_2)M_2^2 +  
        (3 + 4\delta_1 + 2\delta_1^2 + \delta_2) M_2 M_1^2  - (1 + 2\delta_1^2 + 2\delta_1)M_1^4.}
\end{align*}
It appears that each coefficient between parentheses is a polynomial in $\delta_1,\delta_2,\dots$ with positive coefficients.
This property will be a consequence of our general formula for $C_n$.

To present the multiplicative extension of Yoshida's weight, we first need some definitions.
If $B\subset\mathbb{N}$ is finite, there is a natural notion of noncrossing partitions of $B$, using the same condition as in the 
definition of $\NC_n$ (the only property that we need is the total order on $B$). They form a lattice denoted $\NC_B$.
The unique order preserving bijection $B\to\{1, \dots , \# B \}$ induces a bijection $\std \,:\, \NC_B \to \NC_n$, called {\it standardization}.
If $\pi\in\NC_B$, its weight is defined as $\wt(\pi) = \wt(\std(\pi))$. 

Also, if $\pi,\rho\in\NC_n$ with $\pi\leq\rho$ and $B\in\rho$, we define the {\it restriction} of $\pi$ to $B$ as:
$\pi|_B = \{ C\in \pi \, : \, C \subset B \}\in\NC_B$. More generally, $\pi|_B \in \NC_B$ is well defined as soon as
$B$ is the union of some blocks of $\pi$.

\begin{defi} \label{defzeta}
The map $\zeta$ on $\NC_n ^2$ is given by: 
\begin{align}
  \zeta( \pi , \rho ) = 
  \begin{cases}
    \prod\limits_{B\in\rho} \wt(\pi|_B) & \text{ if } \pi\leq\rho, \\
    0                            & \text{ otherwise}.
  \end{cases}
\end{align}
\end{defi}

It is a refinement by the parameters $\delta_1,\delta_2, \dots$ of the poset theoretic $\zeta$ function of $\NC_n$.

\begin{prop}
If $\pi\leq\rho$, we have:
\begin{align} \label{defzeta2}
 \zeta(\pi,\rho) 
 = 
 \prod\limits_{ (i,j) \in \arc(\pi) } \delta_{\#\big\{ k \, : \, i<k<j, \text{ and } i \overset{\rho}{\sim} k \overset{\rho}{\sim} j \big\}}.
\end{align}
\end{prop}

\begin{proof}
Let us first show that for any finite $B\subset\mathbb{N}$ and $\pi\in\NC_B$, we have:
\begin{equation} \label{wtpib}
  \wt( \pi ) = \prod_{(i,j) \in \arc(\pi)} \delta_{\#\{ k\in B \,:\, i<k<j \}}.
\end{equation}
If $B = \{1,2,\dots, \#B \}$, we have $\#\{ k\in B \,:\, i<k<j \} = j-i-1$ and we recover the definition of the weight.
The right hand side of \eqref{wtpib} is clearly unchanged by the standardization process, so we get \eqref{wtpib} in general.

Let $\pi\leq \rho$ in $\NC_n$, then we have:
\[
  \zeta( \pi , \rho ) =  \prod\limits_{B\in\rho} \wt(\pi|_B) 
  = 
  \prod_{B\in\rho} \prod_{\substack{ (i,j) \in \arc(\pi) \\ i,j \in B }} \delta_{\#\{ k\in B \,:\, i<k<j \}},
\]
and we get \eqref{defzeta2}.
\end{proof}

\begin{lemm}
We have:
\begin{equation} \label{relMzetaV}
  M_{\rho} = \sum_{\substack{ \pi \in \NC_n \\  \pi \leq \rho  }}   C_{\pi}  \zeta(\pi,\rho), \qquad \text{ where } M_{\rho} = \prod_{B\in\rho} M_{\#B}.
\end{equation}
\end{lemm}

\begin{proof}
Using \eqref{relMC} with $\NC_B$ instead of $\NC_{\#B}$, we can write:
\[
  M_{\rho} = \prod_{B\in\rho} M_{\#B} = \prod_{B\in\rho} \Big( \sum_{\pi\in\NC_B} C_{\pi} \wt(\pi) \Big).
\]
Then we expand the product. Using the fact that the map $\pi \mapsto (\pi|_B)_{B\in\rho}$ 
is an order preserving bijection from $\{ \pi\in\NC_n \, : \, \pi \leq \rho \}$ to $\prod_{B\in\rho} \NC_B$,
and the definition of $\zeta$ as a product of weights, we get the announced formula.
\end{proof}

We can see $\zeta$ as a matrix whose rows and columns are indexed by $\NC_n$, and define its inverse $\mu = \zeta^{-1}$.
It is a refinement by the parameters $\delta_1,\delta_2, \dots$ of the Möbius function of $\NC_n$.
It follows from \eqref{relMzetaV} that:
\begin{align}  \label{cmmu}
  C_{1_n} = \sum_{ \pi \in \NC_n } M_{\pi}  \mu(\pi,1_n).
\end{align}
So it remains to make $\mu(\pi,1_n)$ explicit. 
To this end, we use some definitions taken from \cite{josuatvergesmenousnovellithibon}.
Schröder trees themselves are classical objects in combinatorics but it was showned there
that they are an alternative to noncrossing partitions for dealing with free cumulants.

\begin{defi}[cf. \cite{josuatvergesmenousnovellithibon}]
Let $\ST_n$ denote the set of {\it Schröder trees} with $n+1$ leaves, defined as plane trees where each 
internal vertex has at least $2$ descendants. Among edges issued from an internal vertex, we have a
{\it left edge}, a {\it right edge}, and other ones are called {\it middle edges}.
Let $\PST_n \subset \ST_n $ denote the set of {\it prime} Schröder trees, defined as those such that
the right edge issued from the root is attached to a leaf.
Also let $\intern(T)$ denote that set of internal vertices of a tree $T$.
\end{defi}

When drawing a tree, we take the convention that all leaves are at the same level.
For example, the Schröder trees with $3$ leaves are:
\begin{align}  \label{schoder3leaves}  
  \begin{tikzpicture}
      \tikzstyle{ver} = [circle, draw, fill, inner sep=0.5mm]
      \tikzstyle{edg} = [line width=0.6mm,<->,>=round cap,join=round]
      \draw[edg] (1,0) -- (2,1) -- (3,0);
      \draw[edg] (2,0) -- (2,1);
  \end{tikzpicture}
  , \quad
  \begin{tikzpicture}
      \tikzstyle{ver} = [circle, draw, fill, inner sep=0.5mm]
      \tikzstyle{edg} = [line width=0.6mm,<->,>=round cap,join=round]
      \draw[edg] (1,0) -- (2,1) -- (3,0);
      \draw[edg] (2,0) -- (1.5,0.5);
  \end{tikzpicture}
  , \quad
  \begin{tikzpicture}
      \tikzstyle{ver} = [circle, draw, fill, inner sep=0.5mm]
      \tikzstyle{edg} = [line width=0.6mm,<->,>=round cap,join=round]
      \draw[edg] (1,0) -- (2,1) -- (3,0);
      \draw[edg] (2,0) -- (2.5,0.5);
  \end{tikzpicture},
\end{align}
and the first $2$ only are prime. Those with $4$ leaves are:
\begin{align} \label{schoder4leaves}
  \tikzset{every picture/.style={scale=0.6}  }
  \begin{tikzpicture}
      \tikzstyle{ver} = [circle, draw, fill, inner sep=0.5mm]
      \tikzstyle{edg} = [line width=0.6mm,<->,>=round cap,join=round]
      \draw[edg] (1,0) -- (2.5,1.5) -- (4,0);
      \draw[edg] (2,0) -- (2.5,1.5) -- (3,0);
  \end{tikzpicture}
  , &\quad
  \begin{tikzpicture}
      \tikzstyle{ver} = [circle, draw, fill, inner sep=0.5mm]
      \tikzstyle{edg} = [line width=0.6mm,<->,>=round cap,join=round]
      \draw[edg] (1,0) -- (2.5,1.5) -- (4,0);
      \draw[edg] (2.5,1.5) -- (3,0);
      \draw[edg] (2,0) -- (1.5,0.5);
  \end{tikzpicture}
  , \quad
  \begin{tikzpicture}
      \tikzstyle{ver} = [circle, draw, fill, inner sep=0.5mm]
      \tikzstyle{edg} = [line width=0.6mm,<->,>=round cap,join=round]
      \draw[edg] (1,0) -- (2.5,1.5) -- (4,0);
      \draw[edg] (2,0) -- (2,1) -- (3,0);
      %\draw[edg] (2.5,0.5) -- (2,1);
  \end{tikzpicture}
  , \quad
  \begin{tikzpicture}
      \tikzstyle{ver} = [circle, draw, fill, inner sep=0.5mm]
      \tikzstyle{edg} = [line width=0.6mm,<->,>=round cap,join=round]
      \draw[edg] (1,0) -- (2.5,1.5) -- (4,0);
      \draw[edg] (2,0) -- (2.5,0.5) -- (3,0);
      \draw[edg] (2.5,0.5) -- (2.5,1.5);
  \end{tikzpicture}
  , \quad
  \begin{tikzpicture}
      \tikzstyle{ver} = [circle, draw, fill, inner sep=0.5mm]
      \tikzstyle{edg} = [line width=0.6mm,<->,>=round cap,join=round]
      \draw[edg] (1,0) -- (2.5,1.5) -- (4,0);
      \draw[edg] (2,0) -- (1.5,0.5);
      \draw[edg] (3,0) -- (2,1);
  \end{tikzpicture}
  , \quad
  \begin{tikzpicture}
      \tikzstyle{ver} = [circle, draw, fill, inner sep=0.5mm]
      \tikzstyle{edg} = [line width=0.6mm,<->,>=round cap,join=round]
      \draw[edg] (1,0) -- (2.5,1.5) -- (4,0);
      \draw[edg] (2,0) -- (2.5,0.5) -- (3,0);
      \draw[edg] (2.5,0.5) -- (2,1);
  \end{tikzpicture}
  ,  
  \\
  \begin{tikzpicture}
      \tikzstyle{ver} = [circle, draw, fill, inner sep=0.5mm]
      \tikzstyle{edg} = [line width=0.6mm,<->,>=round cap,join=round]
      \draw[edg] (1,0) -- (2.5,1.5) -- (4,0);
      \draw[edg] (2.5,1.5) -- (2,0);
      \draw[edg] (3,0) -- (3.5,0.5);
  \end{tikzpicture} &, \quad
  \begin{tikzpicture}
      \tikzstyle{ver} = [circle, draw, fill, inner sep=0.5mm]
      \tikzstyle{edg} = [line width=0.6mm,<->,>=round cap,join=round]
      \draw[edg] (1,0) -- (2.5,1.5) -- (4,0);
      \draw[edg] (2,0) -- (3,1);
      \draw[edg] (3,0) -- (3,1);
  \end{tikzpicture} ,\quad
  \begin{tikzpicture}
      \tikzstyle{ver} = [circle, draw, fill, inner sep=0.5mm]
      \tikzstyle{edg} = [line width=0.6mm,<->,>=round cap,join=round]
      \draw[edg] (1,0) -- (2.5,1.5) -- (4,0);
      \draw[edg] (2,0) -- (3,1);
      \draw[edg] (3,0) -- (2.5,0.5);
  \end{tikzpicture} ,\quad
  \begin{tikzpicture}
      \tikzstyle{ver} = [circle, draw, fill, inner sep=0.5mm]
      \tikzstyle{edg} = [line width=0.6mm,<->,>=round cap,join=round]
      \draw[edg] (1,0) -- (2.5,1.5) -- (4,0);
      \draw[edg] (2,0) -- (3,1);
      \draw[edg] (3,0) -- (3.5,0.5);
  \end{tikzpicture} ,\quad
  \begin{tikzpicture}
      \tikzstyle{ver} = [circle, draw, fill, inner sep=0.5mm]
      \tikzstyle{edg} = [line width=0.6mm,<->,>=round cap,join=round]
      \draw[edg] (1,0) -- (2.5,1.5) -- (4,0);
      \draw[edg] (2,0) -- (1.5,0.5);
      \draw[edg] (3,0) -- (3.5,0.5);
  \end{tikzpicture},
\end{align}
and the first $6$ only are prime.

\begin{defi}[cf. \cite{josuatvergesmenousnovellithibon}]
The map $\eta : \mathcal{S}'_n \to \NC_n$ is given by the following rule. Let $T\in\mathcal{T}_n$.
First, we place labels $1,2,\dots,n$ such that $i$ is placed between the $i$th and $(i+1)$st leaves, from left to right.
Then, we have $i \overset{\eta(T)}{\sim} j$ iff we can draw a path from label $i$ to label $j$ that stays above the level of the leaves,
and cross only middle edges of $T$. 
%We denote $\PST_n(\pi)$ the set of $T\in\PST_n$ that projects to the noncrossing partition $\pi$ under this map.
\end{defi}

For example,
\[
  \eta\Big( \;
  \begin{tikzpicture}
      \tikzstyle{ver} = [circle, draw, fill, inner sep=0.5mm]
      \tikzstyle{edg} = [line width=0.6mm,<->,>=round cap,join=round]
      \tikzstyle{pat} = [color=gray,dotted,line width=0.4mm]
      \node      at (1.5,0) {\small 1};
      \node      at (2.5,0) {\small 2};
      \node      at (3.5,0) {\small 3};
      \node      at (4.5,0) {\small 4};
      \node      at (5.5,0) {\small 5};
      \node      at (6.5,0) {\small 6};
      \node      at (7.5,0) {\small 7};
      \pgfresetboundingbox
      \draw[edg] (1,0) -- (4.5,3.5) -- (8,0);
      \draw[edg] (2,0) -- (1.5,0.5);
      \draw[edg] (3,0) -- (5,2) -- (7,0);
      \draw[edg] (5,2) -- (4.5,3.5);
      \draw[edg] (4,0) -- (5,2);
      \draw[edg] (5,0) -- (5.5,0.5) -- (6,0);
      \draw[edg] (5.5,0.5) -- (5,2);
      \draw[pat] (2.6,0.3) .. controls (4.8,3) .. (7.4,0.3);
      \draw[pat] (3.6,0.3) .. controls (4,0.7) .. (4.4,0.3);
      \draw[pat] (4.6,0.3) .. controls (5.5,1.2) .. (6.4,0.3);
  \end{tikzpicture} 
  \;\Big)
  = 1|27|346|5.
\]
Another property that we will need and is elementary to check is that 
\begin{equation}  \label{signbij}
  \# \eta(T) = \# \intern(T).
\end{equation}

\begin{defi}  \label{defiwttree}
The {\it left branch} of a tree $T\in\PST_n$ is the path going from the root down to the leftmost leaf.
Let $\intern'(T)$ denote the set of internal vertices that are not in the left branch of $T$.
The {\it degree} $\deg(v)$ of $v \in \intern(T)$ is its number of descendants.
And the {\it weight} of $T\in\PST_n$ is:
\begin{equation} \label{defwttrees}
   \wt(T) = \prod_{v\in \intern'(T)} \delta_{ \deg(v) -1 }.
\end{equation}
\end{defi}

We have now all necessary definitions to state:
\begin{theo} \label{muschroder}
For any $\pi\in \NC_n$, we have:
\begin{align} \label{muschrodereq}
  \mu(\pi,1_n) = (-1)^{ \#\pi - 1 } \sum_{ T \in \PST_n, \; \eta(T)=\pi } \wt(T).
\end{align}
\end{theo}

This will be proved in the next section. Together with Equations~\eqref{cmmu} and \eqref{signbij}, it immediately follows:
\begin{theo} \label{theotrees}
 The $n$th $\Delta$-cumulant is given combinatorially by
\[
  C_n = \sum_{ T \in \PST_n } M_{\eta(T)} (-1)^{ \#\intern(T) - 1 }  \wt(T).
\]
\end{theo}

For example, one can check that the $6$ trees in \eqref{schoder4leaves} (in this order) gives the formula for $C_3$
given at the beginning of this section.

In the free case ($\delta_i=1$ for all $i$, hence $\wt(T)=1$ for all $T\in\mathcal{S}'_n$), this was obtained 
in \cite{josuatvergesmenousnovellithibon}. It was proved there that this formula in terms of prime Schröder 
trees implies Speicher's one involving the Möbius function of $\NC_n$ \cite{speicher}.

In the boolean case ($\delta_1=1$, $\delta_i=0$ for $i\geq 2$), we have $\wt(T)=1$ if all internal vertices of 
$T$ are in the left branch, and $0$ otherwise. Such trees with $\wt(T)=1$ are in bijection
with interval partitions, via the map $\eta$ (suitably restricted). For example,
\[
  \eta\Big( \;
     \begin{tikzpicture}
      \tikzstyle{ver} = [circle, draw, fill, inner sep=0.5mm]
      \tikzstyle{edg} = [line width=0.6mm,<->,>=round cap,join=round]
      \node      at (1.5,0) {\small 1};
      \node      at (2.5,0) {\small 2};
      \node      at (3.5,0) {\small 3};
      \node      at (4.5,0) {\small 4};
      \node      at (5.5,0) {\small 5};
      \node      at (6.5,0) {\small 6};
      \node      at (7.5,0) {\small 7};
      \pgfresetboundingbox
      \draw[edg] (1,0) -- (4.5,3.5) -- (8,0);
      \draw[edg] (2,0) -- (1.5,0.5);
      \draw[edg] (3,0) -- (3,2);
      \draw[edg] (4,0) -- (3,2);
      \draw[edg] (5,0) -- (3,2);
      \draw[edg] (6,0) -- (4,3);
      \draw[edg] (7,0) -- (4,3);
     \end{tikzpicture} 
   \; \Big) = 1|234|56|7.  
\]
Moreover, the factor $(-1)^{ \#\intern(T) - 1 }$ is easily seen to be the Möbius function of $\IN_n$ evaluated 
at $(\eta(T),1_n)$, so we recover the known formula for boolean cumulants \cite{speicherworoudi}.

\section{\texorpdfstring{Proof of Equation~\eqref{muschrodereq}}{Proof of Equation (22)}}

Let $V_\pi$ denote the right hand side of \eqref{muschrodereq}, and for $\rho\in\NC_n$, let 
\[
  W_{\rho} = \sum_{\substack{ \pi \in \NC_n \\ \rho \leq \pi \leq 1_n }} \zeta(\rho,\pi) V_\pi.
\]
Our goal is to show that $W_{1_n} = 1$ and $W_{\rho}=0$ if $\rho \neq 1_n$.
Indeed, these equations precisely say that $(V_{\pi})_{\pi\in\NC_n}$ is the column 
vector of $\zeta^{-1}$ indexed by $1_n$, i.e.~$V_{\pi} = \mu(\pi,1_n)$.

First note that $W_{1_n}=1$ is straightforward. The sum defining $W_{1_n}$ is reduced to the unique term $\zeta(1_n,1_n) V_{1_n}$.
Moreover $V_{1_n}=1$ because there is a unique $T\in \mathcal{S}'_n$ such that $\eta(T)=1_n$, 
that having one internal vertex whose $n+1$ descendants are the $n+1$ leaves.
So, from now on we assume $\rho<1_n$ and we want to prove $W_{\rho}=0$.

Let us first rewrite the formula for $V_{\pi}$ in terms of other combinatorial objects, also taken from \cite{josuatvergesmenousnovellithibon}.

\begin{defi}
Let $\mathcal{A}_n$ denote the set of {\it noncrossing arrangements of binary trees} with $n$ leaves, defined as follows.
Given $n$ dots on a horizontal axis, $A\in\mathcal{A}_n$ is a set of binary trees such that: each of the $n$ dots is a leaf 
of exactly one of the trees, and edges do not cross when the trees are drawn in the canonical way
(formally described by the fact that the edges issued from an internal vertex go in the South East and South West directions).
Also, for $A\in\mathcal{A}_n$, we define a noncrossing partition $\overline{A}\in\NC_n$ as follows: label the leaves by $1,2,\dots,n$
from left to right, then each block of $\overline{A}$ is the set of labels of the leaves in some tree of $A$.
\end{defi}

For example, an element $A \in \mathcal{A}_{11}$ is in the right part of Figure~\ref{arrtrees}, 
and the associated noncrossing partition is $\overline{A} = 1456|23|78{\rm AB}|9$. 
Note that the map $A\mapsto \overline{A}$ is surjective but not injective.

\begin{figure}[h!tp] \centering \captionsetup{width=.7\linewidth}
 \tikzset{every picture/.style={scale=0.4}}
 \begin{tikzpicture}
   \tikzstyle{ver} = [circle, draw, fill, inner sep=0.5mm]
   \tikzstyle{edg} = [line width=0.6mm,<->,>=round cap,join=round]
%    \node[ver] at (1,0) {};
%    \node[ver] at (2,0) {};
%    \node[ver] at (3,0) {};
%    \node[ver] at (4,0) {};
%    \node[ver] at (5,0) {};
%    \node[ver] at (6,0) {};
%    \node[ver] at (7,0) {};
%    \node[ver] at (8,0) {};
%    \node[ver] at (9,0) {};
%    \node[ver] at (10,0) {};
%    \node[ver] at (11,0) {};
%    \node      at (1.5,-1) {1};
%    \node      at (2.5,-1) {2};
%    \node      at (3.5,-1) {3};
%    \node      at (4.5,-1) {4};
%    \node      at (5.5,-1) {5};
%    \node      at (6.5,-1) {6};
%    \node      at (7.5,-1) {7};
%    \node      at (8.5,-1) {8};
%    \node      at (9.5,-1) {9};
%    \node      at (10.5,-1) {A};
   \draw[edg] (1,0) to (6.5,5.5) to (12,0);
   \draw[edg] (6,0) to (3.5,2.5);   
   \draw[edg] (3,2) to (5,0);   
   \draw[edg] (4.5,0.5) to (4,0);   
   \draw[edg] (2,0) to (2.5,0.5) to (3,0);   
   \draw[edg] (7,0) to (9,2) to (11,0);      
   \draw[edg] (9.5,1.5) to (8,0);
   \draw[edg] (2.5,0.5) to (3,2);   
   \draw[edg] (9,0) to (9.5,1.5);   
   \draw[edg] (10,0) to (10.5,0.5);   
   \draw[edg] (6.5,5.5) to (9,2);   
 \end{tikzpicture}
 \hspace{2.5cm}
 \begin{tikzpicture}
   \tikzstyle{ver} = [circle, draw, fill, inner sep=0.18mm]
   \tikzstyle{edg} = [line width=0.6mm,<->,>=round cap,join=round]
%    \node[ver] at (1,0) {};
%    \node[ver] at (2,0) {};
%    \node[ver] at (3,0) {};
%    \node[ver] at (4,0) {};
%    \node[ver] at (5,0) {};
%    \node[ver] at (6,0) {};
%    \node[ver] at (7,0) {};
%    \node[ver] at (8,0) {};
   \node[ver] at (9,0.06) {};
%    \node[ver] at (10,0) {};
   \node      at (1,-1) {1};
   \node      at (2,-1) {2};
   \node      at (3,-1) {3};
   \node      at (4,-1) {4};
   \node      at (5,-1) {5};
   \node      at (6,-1) {6};
   \node      at (7,-1) {7};
   \node      at (8,-1) {8};
   \node      at (9,-1) {9};
   \node      at (10,-1) {A};
   \node      at (11,-1) {B};
   \draw[edg] (1,0) to (3.5,2.5) to (6,0);
   \draw[edg] (3,2) to (5,0);   
   \draw[edg] (4.5,0.5) to (4,0);   
   \draw[edg] (2,0) to (2.5,0.5) to (3,0);
   \draw[edg] (7,0) to (9,2) to (11,0);
   \draw[edg] (10,0) to (10.5,0.5);
   \draw[edg] (9.5,1.5) to (8,0);
 \end{tikzpicture}
 \caption{The bijection $\phi$. \label{arrtrees}}
\end{figure}
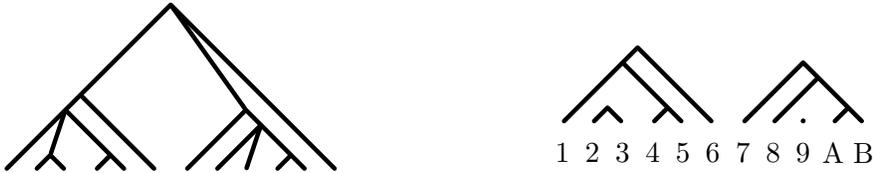

\begin{prop}[cf. \cite{josuatvergesmenousnovellithibon}]
 There is a bijection $\phi:\PST_n \to \mathcal{A}_n$ such that for $T\in\PST_n$, $\phi(T)$ is obtained 
 from $T$ by removing the root and its incident edges, and removing every middle edge of the tree.
\end{prop}

See Figure~\ref{arrtrees} for an example, and note that there is an obvious identification of
the vertices of $\phi(T)$ with vertices of $T$ different from the root and the rightmost leaf.
For $A\in\mathcal{A}_n$ we will denote $\intern'(A) = \intern'(T)$ where 
$T$ is the element of $\PST_n$ such that $\phi(T)=A$.

\begin{defi}
We extend the weight function to $\mathcal{A}_n$ by the rule that $\wt(\phi(T)) = \wt(T)$ for any $T\in\PST_n$.
% In an arrangement $A\in\mathcal{A}_n$, there is a tree $T_0$ containing the leftmost leaf, then 
% the {\it left branch} of $A$ is defined as the path from the root of $T_0$ to the leftmost leaf.
For two internal vertices $v_1,v_2\in\intern(A)$, we say that $v_1$ {\it covers} $v_2$ if, 
in the unique $T\in\PST_n$ such that $\phi(T)=A$, $v_2$ is a descendant of $v_1$ via a middle edge.
For $v\in\intern'(A)$, we denote by $\cov(v)$ the number of vertices covered by $v$.
\end{defi}

\begin{lemm} \label{wtarrtrees}
 For $A\in\mathcal{A}_n$, we have:
\begin{equation} \label{wttrees}
  \wt(A) = \prod_{v\in\intern'(A)} \delta_{ \cov(v)+1 }.
\end{equation}
\end{lemm}

\begin{proof}
 This is a simple reformulation of Definition~\ref{defiwttree} using the bijection $\phi$.
 Note that for $v\in\intern'(T)$, the number of vertices it covers is $\deg(v)-2$, 
 since these are all its descendants except the left and right ones. This explains why
 the index $\deg(v)-1$ in \eqref{defwttrees} becomes $\cov(v)+1$ here.
\end{proof}

To state the next lemma, we need the classical notion of {\it Kreweras complement} \cite{kreweras}.
Let $\pi\in\NC_n$. Suppose we have $2n$ labels $1$, $1'$, $2$, $2'$, $\dots$, $n$, $n'$ on a horizontal
line, in this order, and that $\pi$ is drawn as in \eqref{ncexample} (only using the labels $1,\dots,n$).
Then the Kreweras complement $\pi^c$ of $\pi$ is defined by the condition that $i \overset{\pi^c}{\sim}j$
iff we can connect $i'$ to $j'$ by a path that stays above the level of the labels, and do not cross the
arches of $\pi$. For example, Figure~\ref{exkrew} shows that $(134|2|59|678|A)^c = 12|3|49A|58|6|7 $.
The map $\pi\mapsto\pi^c$ is a poset anti-isomorphism from $\NC_n$ to itself, and its inverse is 
denoted $\pi \mapsto \mathbin{^c\pi}$. We refer to \cite{kreweras} for details.

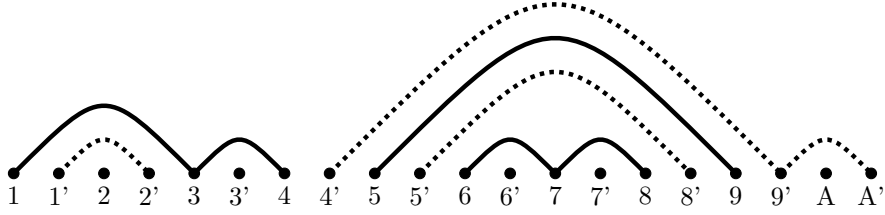
\begin{figure} \centering
 \begin{tikzpicture}
   \tikzstyle{ver} = [circle, draw, fill, inner sep=0.5mm]
   \tikzstyle{edg} = [line width=0.6mm]
   \tikzstyle{edg2} = [line width=0.6mm,dotted]
   \node      at (1,-0.5) {\small 1};
   \node      at (2,-0.5) {\small 1'};
   \node      at (3,-0.5) {\small 2};
   \node      at (4,-0.5) {\small 2'};
   \node      at (5,-0.5) {\small 3};
   \node      at (6,-0.5) {\small 3'};
   \node      at (7,-0.5) {\small 4};
   \node      at (8,-0.5) {\small 4'};
   \node      at (9,-0.5) {\small 5};
   \node      at (10,-0.5) {\small 5'};
   \node      at (11,-0.5) {\small 6};
   \node      at (12,-0.5) {\small 6'};
   \node      at (13,-0.5) {\small 7};
   \node      at (14,-0.5) {\small 7'};
   \node      at (15,-0.5) {\small 8};
   \node      at (16,-0.5) {\small 8'};
   \node      at (17,-0.5) {\small 9};
   \node      at (18,-0.5) {\small 9'};
   \node      at (19,-0.5) {\small A};
   \node      at (20,-0.5) {\small A'};
%   \pgfresetboundingbox
   \node[ver] at (1,0) {};
   \node[ver] at (2,0) {};
   \node[ver] at (3,0) {};
   \node[ver] at (4,0) {};
   \node[ver] at (5,0) {};
   \node[ver] at (6,0) {};
   \node[ver] at (7,0) {};
   \node[ver] at (8,0) {};
   \node[ver] at (9,0) {};
   \node[ver] at (10,0) {};
   \node[ver] at (11,0) {};
   \node[ver] at (12,0) {};
   \node[ver] at (13,0) {};
   \node[ver] at (14,0) {};
   \node[ver] at (15,0) {};
   \node[ver] at (16,0) {};
   \node[ver] at (17,0) {};
   \node[ver] at (18,0) {};
   \node[ver] at (19,0) {};
   \node[ver] at (20,0) {};
   \draw[edg] (1,0) .. controls (3,2) .. (5,0);
   \draw[edg] (5,0) .. controls (6,1) .. (7,0);
   \draw[edg] (9,0) .. controls (13,4) .. (17,0);  
   \draw[edg] (11,0) .. controls (12,1) .. (13,0);  
   \draw[edg] (13,0) .. controls (14,1) .. (15,0);  
   \draw[edg2] (2,0) .. controls (3,1) .. (4,0);  
   \draw[edg2] (8,0) .. controls (13,5) .. (18,0);  
   \draw[edg2] (10,0) .. controls (13,3) .. (16,0);  
   \draw[edg2] (18,0) .. controls (19,1) .. (20,0);  
 \end{tikzpicture}
 \caption{Kreweras complementation.\label{exkrew}} 
\end{figure}

\begin{lemm}[cf. \cite{josuatvergesmenousnovellithibon}]
 If $T\in\mathcal{S}'_n$, we have $\overline{\phi(T)}^c = \eta(T)$.
\end{lemm}

Using the bijection $\phi$ and the previous lemma, we can write $V_{\pi}$ in terms of 
noncrossing arrangements of binary trees:
\[
  V_{\pi} = (-1)^{\#\pi-1} \sum_{\substack{ A \in \mathcal{A}_n \\  \overline{A}^c = \pi }} \wt(A).
\]
A property of Kreweras complementation is that $\#\pi + \# \pi^c = n+1$.
Note also that we have clearly $\#\overline{A} = \#A$ for $A\in\mathcal{A}_n$.
So $(-1)^{\#\pi-1} = (-1)^{n-\#A}$ if $\overline{A}^c = \pi$.
Plugging the previous formula for $V_{\pi}$ in the definition of $W_{\rho}$, it follows:
\begin{align*}
  (-1)^n W_{\rho} = \sum_{\substack{ \pi\in\NC_n \\ \rho \leq \pi \leq 1_n }} \zeta( \rho, \pi ) \sum_{\substack{ A \in \mathcal{A}_n \\  \overline{A}^c = \pi }} (-1)^{\# A} \wt(A)
           = \sum_{\substack{ A \in \mathcal{A}_n \\  \rho\leq \overline{A}^c }} \zeta( \rho, \overline{A}^c ) (-1)^{\# A} \wt(A).
\end{align*}
Kreweras complementation being a poset anti-automorphism, we can change the condition in the summation to get:
\begin{align*}
  (-1)^n W_{\rho} = \sum_{\substack{ A \in \mathcal{A}_n \\  \overline{A} \leq \mathbin{^c\rho} }} \zeta( \rho, \overline{A}^c ) (-1)^{\# A} \wt(A).
\end{align*}
Then, let us define a map $\zeta^c$ by $\zeta^c ( \alpha, \beta ) = \zeta(\beta^c,\alpha^c)$.
Here we exchange the arguments to keep the fact that $\zeta^c ( \alpha, \beta ) = 0$ if $\alpha\nleq\beta$, just as $\zeta$.
We get the following equality:
\begin{align} \label{wrho}
  (-1)^n W_{\rho} = \sum_{\substack{ A \in \mathcal{A}_n \\  \overline{A} \leq \mathbin{^c\rho} }} \zeta^c( \overline{A} , \mathbin{^c\rho} ) (-1)^{\# A} \wt(A).
\end{align}

% Let $X,Y$ be two set of positive integers, we say that $X$ is below $Y$ if $\min(Y) \leq \min(X)$ and $\max(X) \leq \max(Y)$, and 
% denote this by $ X \preccurlyeq Y $.

We will show that this quantity is $0$ by pairing terms, but we need another lemma before doing that.

If $B\subset\mathbb{N}$ is finite, we denote $[B]$ the smallest interval containing $B$, 
i.e.~the set of consecutive integers $\{\min(B) , \min(B)+1 , \dots, \max(B)\}$. Note that if $\pi\in\NC_n$ and $B\in\pi$, 
$[B]$ is the union of some blocks of $\pi$.

If $\pi\in\NC_n$, there is an interval partition which is minimal among interval partitions
above $\pi$, and its number of blocks is denoted $\iota(\pi)$. It is easily seen that this number can be computed as 
follows: consider $B_1\in\pi$ with $\min(B_1)=1$, then $B_2\in\pi$ with $\min(B_2) = \max(B_1)+1$, and so on until
we find $B_k$ with $\min(B_k) = \max(B_{k-1})+1$, and $\max(B_k)= n$, this last condition meaning that 
$B_{k+1}$ cannot be defined and the process stops. Then $k = \iota(\pi)$. More precisely the smallest interval partition
above $\pi$ is $\{[B_1],\dots,[B_k]\}$.

We also extend this map $\iota$
to $\NC_B$ if $B\subset\mathbb{N}$ by the requirement $\iota(\pi)=\iota(\std(\pi))$.

\begin{lemm}
If $\alpha\leq \beta$, we have:
\begin{equation} \label{defzetac2}
  \zeta^c( \alpha, \beta ) = \prod_{\substack{ B \in \beta, \\  1 \notin B \text{ and } \min(B) \overset{\alpha}{\nsim} \max(B) }}  
     \delta_{ \iota(\alpha|_{[B]}) -1 }.
\end{equation}
%where $f(\alpha,B)$ is the number of blocks $C \in \alpha$ such that: $C\preccurlyeq B$, and there is no $C'\in\alpha$ such that $C\preccurlyeq C' \preccurlyeq B$.
%and there is no $C'\in\alpha$ such that $C'\subset B$ and $C$ is nested from above by $C'$.
\end{lemm}

\begin{proof}
We will use the following fact, which is straightforward from the definition of Kreweras complementation: 
asumming $1\leq i < j \leq n$, 
$(i,j)$ is an arch of $\pi^c$ if and only if there is a block $B\in \pi$ such that $\min(B) = i+1 $ and $\max(B)=j$.

Our goal is as follows: to each factor $\delta_k$ in $\zeta(\beta^c,\alpha^c)$, 
associate a factor $\delta_k$ in the right hand side of \eqref{defzetac2}, and reciprocally.

Such a factor $\delta_k$ in $\zeta(\beta^c,\alpha^c)$ means we can find $j_1,\dots,j_{k+2}$ such 
that $1\leq j_1<j_2 < \dots < j_{k+2}\leq n$,
$(j_1,j_{k+2}) \in \arc( \beta^c )$, and $(j_1,j_2), \dots , (j_{k+1},j_{k+2}) \in \arc( \alpha^c )$.
This follows from Equation~\eqref{defzeta2}.

From $(j_1,j_{k+2}) \in \arc( \beta^c )$, we get that $\beta$ contains a block $B$ with $\min(B) = j_1+1$, and 
$\max(B) = j_{k+2}$. Similarly, there exist $B_1,\dots, B_{k+1} \in \alpha$ such that $\min(B_i) = j_i+1$ and 
$\max(B_i) = j_{i+1}$ (for $1\leq i \leq k+1$).

This block $B$ shows that there is a factor $\delta_k$ in the right hand side of \eqref{defzetac2}.
Indeed, we have $1\notin B$ since $\min(B) = j_1+1\geq 2$. We have $\min(B)\in B_1$ and
$\max(B) \in B_{k+1}$ so $\min(B) \overset{\alpha}{\nsim}\max(B)$.
The sets $B_1,\dots,B_{k+1}$ are blocks of $\alpha|_{[B]}$, and the relations between their maxima and minima 
show that $\iota(\alpha|_{[B]})=k+1$. So we get a factor $\delta_k$ in the right hand side of \eqref{defzetac2}, 
as needed.

In the other direction, we can check that starting from $B$ and the blocks $B_1,\dots,B_{k+1}$, 
we find $j_1,\dots,j_{k+2}$ as above.
\end{proof}

The next step is to define a fixed point free involution $\Psi$ on the set $\{ A\in\mathcal{A}_n \, : \, \overline{A} \leq \mathbin{^c\rho} \}$,
such that
\[
   \zeta^c( \overline{A} , \mathbin{^c\rho} ) (-1)^{\# A} \wt(A) = - \zeta^c( \overline{\Psi(A)} , \mathbin{^c\rho} ) (-1)^{\# \Psi(A)} \wt(\Psi(A)).
\]
It will show that the right hand side of Equation~\eqref{wrho} vanishes, since terms indexed 
by $A$ and $\Psi(A)$ cancel each other out, hence it will complete the proof of Equation~\eqref{muschrodereq}.

To begin, we denote $B_0$ the (unique) block of $\mathbin{^c\rho}$ such that $\# B_0 \geq 2$, and $\min(B_0) < \min(B) $ if $B$ is another 
block of $\mathbin{^c\rho}$ such that $\# B \geq 2$. Since $\rho \neq 1_n$, we have $\mathbin{^c\rho} \neq 0_n $, so $B_0$ exists.
To define $\Psi(A)$, we distinguish two cases, whether $\min(B_0) \overset{ \overline{A} }{\sim}  \max(B_0) $ or not.

\begin{itemize}
 \item If $\min(B_0) \overset{ \overline{A} }{\sim}  \max(B_0) $, there is a tree $T$ in the arrangement $A$, two of whose leaves are labelled 
       by $\min(B_0)$ and $\max(B_0)$. Let $v_0$ denote the root of $T$.
       Then, $\Psi(A)$ is defined by removing $v_0$ (as well as the two edges issued from it).
 \item In the other case, $\min(B_0) \overset{ \overline{A} }{\nsim}  \max(B_0) $, it is the reverse operation.
       Let $T_1$ and $T_2$ be the trees in $A$ that respectively contain $\min(B_0)$ and $\max(B_0)$. Then $\Psi(A)$ is obtained from $A$
       by adding a new internal vertex $v$, whose two descendants are the roots of $T_1$ and $T_2$.        
\end{itemize}

\tikzset{every picture/.style={scale=0.4}  }
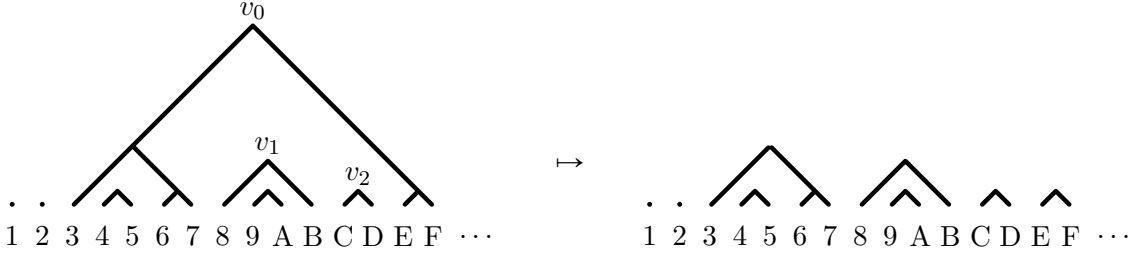
\begin{figure} \centering
 \begin{tikzpicture} 
   \tikzstyle{ver} = [circle, draw, fill, inner sep=0.18mm]
   \tikzstyle{edg} = [line width=0.6mm,<->,>=round cap,join=round]
   \node[ver] at (1,0.06) {};
   \node[ver] at (2,0.06) {};
   \node      at (1,-1) {1};
   \node      at (2,-1) {2};
   \node      at (3,-1) {3};
   \node      at (4,-1) {4};
   \node      at (5,-1) {5};
   \node      at (6,-1) {6};
   \node      at (7,-1) {7};
   \node      at (8,-1) {8};
   \node      at (9,-1) {9};
   \node      at (10,-1) {A};
   \node      at (11,-1) {B};
   \node      at (12,-1) {C};
   \node      at (13,-1) {D};
   \node      at (14,-1) {E};
   \node      at (15,-1) {F};
   \node      at (16.5,-1) {$\dots$};
%   \node      at (16.5,1) {$\dots$};
   \node      at (9,6.5) {$v_0$};
   \node      at (9.5,2) {$v_1$};
   \node      at (12.5,1) {$v_2$};
   \draw[edg] (3,0) to (9,6) to (15,0);
   \draw[edg] (4,0) to (4.5,0.5) to (5,0);
   \draw[edg] (6,0) to (6.5,0.5);
   \draw[edg] (7,0) to (5,2);
   \draw[edg] (8,0) to (9.5,1.5) to (11,0);
   \draw[edg] (9,0) to (9.5,0.5) to (10,0);
   \draw[edg] (12,0) to (12.5,0.5) to (13,0);
   \draw[edg] (14,0) to (14.5,0.5);
 \end{tikzpicture}
 \quad
 \begin{tikzpicture}  
 \draw[color=white] (0,0) rectangle (1,6);
 \node at (0.5,3) {$\mapsto$};
 \end{tikzpicture}
 \quad
 \begin{tikzpicture} 
   \tikzstyle{ver} = [circle, draw, fill, inner sep=0.18mm]
   \tikzstyle{edg} = [line width=0.6mm,<->,>=round cap,join=round]
   \node[ver] at (1,0.06) {};
   \node[ver] at (2,0.06) {};
   \node      at (1,-1) {1};
   \node      at (2,-1) {2};
   \node      at (3,-1) {3};
   \node      at (4,-1) {4};
   \node      at (5,-1) {5};
   \node      at (6,-1) {6};
   \node      at (7,-1) {7};
   \node      at (8,-1) {8};
   \node      at (9,-1) {9};
   \node      at (10,-1) {A};
   \node      at (11,-1) {B};
   \node      at (12,-1) {C};
   \node      at (13,-1) {D};
   \node      at (14,-1) {E};
   \node      at (15,-1) {F};
   \node      at (16.5,-1) {$\dots$};
%   \node      at (16.5,1) {$\dots$};
   \draw[edg] (3,0) to (5,2);
   \draw[edg] (4,0) to (4.5,0.5) to (5,0);
   \draw[edg] (6,0) to (6.5,0.5);
   \draw[edg] (7,0) to (5,2);
   \draw[edg] (8,0) to (9.5,1.5) to (11,0);
   \draw[edg] (9,0) to (9.5,0.5) to (10,0);
   \draw[edg] (12,0) to (12.5,0.5) to (13,0);
   \draw[edg] (14,0) to (14.5,0.5) to (15,0);
 \end{tikzpicture}
 \caption{The involution $\Psi$.\label{defpsi}}
\end{figure}

To check that we can add the two new edges without creating a crossing in the latter case, observe that 
since $\overline{A} \leq \mathbin{^c\rho}$, and $B_0 \in \mathbin{^c\rho}$, there exists a noncrossing 
partitions obtained from $\overline{A}$ obtained by merging the block containing $\min(B_0)$ with that containing $\max(B_0)$.
This shows that $\Psi$ is a well-defined pairing on the set $\{ A\in\mathcal{A}_n \,:\, \overline{A} \leq \mathbin{^c\rho} \}$. 
An example is given in Figure~\ref{defpsi}, with $B_0 = 3678{\rm BEF}$ (for example).

Also, the number of trees in $\Psi(A)$ is one more or one less than that of $A$, so $(-1)^{\#A} = - (-1)^{\#\Psi(A)}$.
It remains only to show:
\begin{equation} \label{zetawt}
  \zeta^c( \overline{A} , \mathbin{^c\rho} ) \wt(A) = \zeta^c( \overline{\Psi(A)} , \mathbin{^c\rho} ) \wt(\Psi(A)).
\end{equation}
Indeed, $\Psi$ has then all the required properties to show that the right hand side of Equation~\eqref{wrho} vanishes.

We can assume that we are in the first case above, i.e.~$\min(B_0) \overset{ \overline{A} }{\sim}  \max(B_0) $, since the
two cases are exchanged under the involution $\Psi$.

First, if $1\in B_0$, we have:
\[
  \zeta^c( \overline{A} , \mathbin{^c\rho} ) = \zeta^c( \overline{\Psi(A)} , \mathbin{^c\rho} ) .
\]
Indeed, in the product of \eqref{defzeta2}, $B_0$ does not appear since it contains $1$, and the other factors cannot change.
We also have:
\[
  \qquad \wt(A) =  \wt(\Psi(A)).
\]
Indeed, $v_0$ is in the left branch of $A$, so we can remove it without changing the product in \eqref{wttrees}.
So \eqref{zetawt} holds.

Now suppose $1\notin B_0$. We have from \eqref{defzetac2}:
\[
  \zeta^c( \overline{\Psi(A)} , \mathbin{^c\rho} )
  =
  \delta_{ \iota\left(\overline{\Psi(A)} |_{[B_0]} \right) -1 } \zeta^c( \overline{A} , \mathbin{^c\rho} ). 
\]
On the other side, we have:
\[
  \delta_{\cov(v_0)+1} \wt(\Psi(A)) = \wt(A).
\]
Multiplying the previous two equations gives \eqref{zetawt},
at the condition that
\begin{equation}  \label{eqcoviota}
  \cov(v_0) + 1 =  \iota\left( \overline{\Psi(A)} |_{[B_0]} \right) - 1.
\end{equation}
This is therefore the last equality to check to complete the proof of the required properties of $\Psi$, hence of $W_{\rho}=0$.

To prove \eqref{eqcoviota}, let us first check on the example of Figure~\ref{defpsi}. The vertices covered by $v_0$ are $v_1$ and $v_2$,
and the smallest interval partition above $[B_0]$ is $34567|89{\rm AB}|{\rm CD}|{\rm EF}$, so \eqref{eqcoviota} holds.
In general, let $v_0'$ and $v_0''$ be the two descendants of $v_0$.
Then, for each vertex $v$ which is either  $v_0'$ or $v_0''$ or covered by $v_0$, 
consider the tree $T$ of $A$ containing $v$, then denote $B_v$ the set of its leaf labels.
Then it is straightforward to see that the intervals $[B_v]$ form the smallest
interval partition above $\overline{\Psi(A)} | _{[B_0]} $. This proves \eqref{eqcoviota}.

% \begin{lemm}
%  Let $A$ be an arrangement of binary trees. For each vertex $v$, let $I_v$ denote the internal
%  
%  Let $v_1, \dots, v_j$ denote the descendants of $v$.
%  Then 
%  
%  
% \end{lemm}
% 
% 

\setlength{\parindent}{0mm}

\end{document}